\documentclass[12pt,reqno]{amsart}

\usepackage{a4wide}
\usepackage{amsmath}
\usepackage{psfrag}
\usepackage{amsthm}
\usepackage{amssymb,amsxtra}
\usepackage{appendix}
\usepackage[usenames]{color}
\usepackage{stmaryrd}
\usepackage{mathrsfs}
\usepackage{upgreek}
\usepackage{ytableau}
\usepackage{arydshln}
\usepackage{comment}
\usepackage{tikz-cd}

\usepackage{paralist}
\usepackage{graphicx}
\usepackage[all]{xy}
\usepackage{tikz}
\usetikzlibrary{matrix,fit}
\usepackage{geometry}
\usepackage{hyperref}
\usepackage{cleveref}

 \allowdisplaybreaks

\newcommand{\NN}{\mathbb{N}}
\newcommand{\Z}{\mathbb{Z}}

\newcommand{\wt}{\widetilde}

\newcommand{\sym}[1]{\mathfrak{S}_{#1}}

\renewcommand{\P}{\Lambda^+}
\newcommand{\pReg}{\Lambda^+_p}
\newcommand{\ccl}[1]{\mathscr{C}_{#1}}
\newcommand{\Ccl}[2]{\mathcal{C}_{#1,#2}}

\newcommand{\B}{\mathbb{B}}

\newcommand{\f}{r}

\newcommand{\C}{\Lambda}
\newcommand{\Fun}{\mathscr{F}}
\newcommand{\F}{{F}}
\newcommand{\FF}{\F}

\newcommand{\A}{\mathbb{A}}

\newcommand{\Ext}{\mathsf{Ext}}

\newcommand{\sol}{\theta}

\newcommand{\cont}{\text{{\tiny $\#$}}}

\newcommand{\im}{\mathrm{im}}

\newcommand{\sref}{\vDash}
\newcommand{\wref}{\preccurlyeq}

\newcommand{\Char}{\mathrm{char}}

\newcommand{\Q}{\mathbb{Q}}
\newcommand{\Fp}{\mathbb{F}_p}
\newcommand{\Ind}{\mathrm{ind}}

\newcommand{\am}{n}

\newcommand{\rad}{{\mathsf{Rad}}}
\newcommand{\Des}[2]{\mathscr{D}_{#1,#2}}

\newcommand{\comp}[1]{#1^*}

\newcommand{\Hom}{{\mathsf{Hom}}}
\newcommand{\soc}{\mathsf{Soc}}

\newcommand{\tbl}[1]{\textcolor{blue}{#1}}
\newcommand{\KJ}[1]{\textcolor{red}{#1}}

\newcommand{\cO}{\mathcal{O}}

\newcommand{\cD}{\mathcal{D}}
\newcommand{\bQ}{\mathbb{Q}}

\theoremstyle{definition}

\newtheorem{defn}{Definition}[section]
\newtheorem{rem}[defn]{Remark}

\newtheorem{eg}[defn]{Example}
\newtheorem{conj}[defn]{Conjecture}

\theoremstyle{plain}

\newtheorem{lem}[defn]{Lemma}
\newtheorem{cor}[defn]{Corollary}
\newtheorem{thm}[defn]{Theorem}
\newtheorem*{mainthm}{Theorem 1.1}

\numberwithin{equation}{section}

\allowdisplaybreaks

\begin{document}
\title{The Representation Type of the Descent Algebras of Type $\A$}

\author{Karin Erdmann}
\address[K. Erdmann]{Mathematical Institute, University of Oxford, Radcliffe Observatory Quarter, Oxford OX2 6GG, United Kingdom.}
\email{erdmann@maths.ox.ac.uk}

\author{Kay Jin Lim}
\address[K. J. Lim]{Division of Mathematical Sciences, Nanyang Technological University, SPMS-04-01, 21 Nanyang Link, Singapore 637371.}
\email{limkj@ntu.edu.sg}

\begin{abstract} We classify the representation type of the descent algebras of type $\A$ in the positive characteristic case. The algebras have finite representation type only for a few small degrees; otherwise, they are wild. Our main reduction method relies on a surjective algebra homomorphism from a descent algebra of type $\A$ to another of lower degree. For small degree cases, we employ techniques from the representation theory of finite-dimensional algebras.
\end{abstract}



\maketitle

\section{Introduction}

The descent algebras of finite Coxeter groups were introduced by Solomon \cite{Solomon:1976a} in 1976. Since then, they have been studied and
exploited in various contexts, in algebra
and algebraic combinatorics and also in
probability theory, discrete geometry and topology.
Let $(W, S)$ be a finite Coxeter system, for example for type $\A_{n-1}$
the group $W$ is the symmetric group
$\sym{n}$. For each subset of $S$ there is a parabolic subgroup of $W$
and a  corresponding permutation character for $W$. The Mackey formula
provides a tool for multiplying these characters and gives rise to a simple
product formula. In \cite{Solomon:1976a},  Solomon discovered a
non-commutative analogue of this multiplication rule, which defines
the descent algebra $\mathscr{D}_\F(W)$ as an explicit subalgebra of the group
algebra $\F W$. He worked with $\F = \Z$, subsequently more
general coefficients were used.
The descent algebra is basic, that is, when $\F$ is a field
its simple modules are one-dimensional. This is rather unusual for naturally
occurring algebras but it makes it feasible to compute explicit sets of primitive
orthogonal idempotents, with potential for applications to
the representation theory of the group algebra $\F W$.

We describe some applications and connections.
To any $W$, there is associated a hyperplane arrangement, and $W$ acts on the
faces of this arrangement. This gives an action of $W$ on the face
semigroup algebra of the arrangement. It was shown by
Bidigare \cite{Bidigare:1997a}  that the subalgebra of invariants is
anti-isomorphic to the descent algebra of $W$.
The face semigroup algebras belong to a larger class of algebras called the unital left regular band algebras and the latter fall into another broader class of algebras called the $\mathcal{R}$-trivial monoid algebras (see \cite{Berg/Bergeron/Bhargava/Saliola:2011a,Schocker:2008a}). The largest class of algebras mentioned in the previous sentence includes the $0$-Hecke algebras.
In  $\mathscr{D}_\F(W)$, one may find a set of primitive orthogonal idempotents which also belong to the group algebra $\F W$, and this is a tool for the study of $\F W$-modules more generally. Specifically in the type $\A$ case, both the complete set of non-isomorphic simple modules for
$\Des{n}{\F}:=\mathscr{D}_\F(\sym{n})$ and $\F\sym{n}$ are labelled by the same set  consisting of $p$-regular partitions of $n$, where $p$ is the characteristic of the field $\F$. Furthermore, the descent algebras of type $\A$ have been used to investigate the Lie modules and Lie powers (see, for example, \cite{BDEM,ELT,ES,Lim:2023a,Magnus/Karrass/Solitar:1976a,Reutenauer:1993a,Schoc03}) and found their applications in the classical results in the representation theory of symmetric groups (see \cite{Blessenohl/Schocker:2005a}).


We describe some previous results.
When $\F = \Q$,
 Solomon \cite{Solomon:1976a} exhibited a basis of $\mathscr{D}_\F(W)$, described its Jacobson radical and amongst other things,  proved that the algebra is basic.
These fundamental properties have been extended to the finite field case by Atkinson--van Willigenburg \cite{Atkinson/vanWilligenburg:1997a} for  type $\A$ and subsequently, together with Pfeiffer \cite{Atkinson/Pfeiffer/vanWilligenburg:2002a}, for the general case. Since $\mathscr{D}_\F(W)$ is basic, by a general theorem of Gabriel  \cite{Gabriel:1979a},
it is isomorphic to a quotient of the path algebra of its Ext quiver.
While the group algebra $\F W$ can be semisimple, that is when $p\nmid |W|$, the descent algebra $\mathscr{D}_\F(W)$ is usually not.
To understand its representation theory, it is therefore desirable to study its simple modules, decomposition and Cartan matrices, Ext quiver, and representation type. The simple modules and decomposition matrix $D$ are described by Atkinson--Pfeiffer--van Willigenburg \cite{Atkinson/Pfeiffer/vanWilligenburg:2002a}.
Furthermore, it is shown in that paper that $\wt{C}=D^\top CD$ where $\wt{C}$ and $C$ are the Cartan matrices for the descent algebra over $\Fp$ and $\Q$ respectively. In fact, the Cartan matrices $C$ and $\wt{C}$ depend only on the characteristics of the corresponding fields. By abuse, we continue to write, for example, $\wt{C}$ for the Cartan matrix of the descent algebra over arbitrary fields of characteristic $p$. In the type $\A$ case, in \cite{Blessenohl/Laue:2002a}, Blessenohl--Laue gave a closed formula for the entries of $C$.
As such, in this case, one can compute $\wt{C}$ using both $D$ and $C$. For the Ext quivers, they are known for both types $\A$ and $\B$ unless $p\mid |W|$, and theoretically it is possible to compute $C$ and $D$  for the exceptional types. When $p=\infty$ and in the type $\A$ case, the Ext quivers have been described by Schocker  \cite{Schocker:2004a}, their
structure  essentially follows  from a result of Blessenohl--Laue \cite{Blessenohl/Laue:2002a}. In that paper, Schocker also described the representation type of $\Des{n}{\Q}$, that is, it has finite representation type if $n\leq 5$, otherwise, it is wild.
Using the connection with  face semigroup algebra, Saliola \cite{Saliola:2008a} obtained the Ext quivers for the descent algebras in the case when $p\nmid |W|$ and $W$ is of either type $\A$ or $\B$.  Otherwise, in general the Ext quivers and the representation type for the descent algebras seem to be not known.

In this paper, we study the representation type of the descent algebras in type $\A$ case when $p<\infty$, and we prove the following theorem.

\begin{thm}\label{T: rep type} Assume $F$ is a field of characteristic $p<\infty$. The descent algebra $\Des{n}{\F}$ has finite representation type if and only if either
\begin{enumerate}[(i)]
  \item $p=2$ and $n\leq 3$,
  \item $p=3$ and $n\leq 4$, or
  \item $p\geq 5$ and $n\leq 5$.
\end{enumerate} Otherwise, it has wild representation type.
\end{thm}

Our crucial reduction method uses an extension of a result by Bergeron--Garsia--Reutenauer \cite{Bergeron/Garsia/Reutenauer:1992a}, that is, for $1\leq s\leq n$, there is a surjective algebra
homomorphism from $\Des{n}{\F}$ onto $\Des{n-s}{\F}$.  With this, it suffices to analyse  a few small degrees  $n$.
For these,  we mostly use  general methods of representation theory   to determine the Ext quivers, and in some cases a presentation of the algebra.

Alternatively, our results can also be proved by computations with elements.
 In order to find the Ext quiver of an algebra, it helps  to know its complete set of primitive orthogonal idempotents, which may be of independent interest. In the type $\A$ case and when $p=\infty$, different sets of primitive orthogonal idempotents for the descent algebras have been obtained, for  example, those given by Garsia--Reutenauer \cite{Garsia/Reutenauer:1989a}. When $p<\infty$, in the recent paper \cite{Lim:2023a}, the second author gave a construction for such set of idempotents. The construction has been subsequently generalised to arbitrary finite Coxeter group by Benson--Lim in \cite[Section 8]{BensonLim2025}. While the construction does not offer a closed formula, in practice, one may perform the computation for small cases and obtain the Ext quivers which is sufficient for our purpose.

The rest of the paper is organised as follows. In the next section, we gather together the necessary background and prove some elementary results. In Sections \ref{S:BGRmap} and \ref{S:pullback}, we consider the
map introduced by Bergeron--Garsia--Reutenauer \cite{Bergeron/Garsia/Reutenauer:1992a} and subsequently study the pullback of the simple modules induced by the map. As a consequence, we show that the Ext quiver of $\Des{n-s}{\F}$ appears as a subquiver of the Ext quiver of $\Des{n}{\F}$ where $n\geq s\geq 1$. We devote most of the final section, Section \ref{S: rep type}, to the proof of Theorem \ref{T: rep type}. Supported by our calculation for small cases, we end the paper with a conjecture for the Ext quivers of the type $\A$ descent algebras.

\section{Preliminaries}\label{S: prelim}

 Let $\F$ be an algebraically closed field of characteristic $p$ (either $p=\infty$   or $p$ is finite, where $p=\infty$ means characteristic zero) and $\cO$ be a commutative ring with 1. We remark that, for our purpose in the study of the representation theory of descent algebras, $\F$ may be taken as any arbitrary field. But we will continue to assume that $\F$ is algebraically closed to avoid any technical issues. Also, in both Subsection \ref{SS:reductionp} and Section \ref{S: rep type}, we require that $p<\infty$ and will highlight the extra assumption there. In the case when $p<\infty$, $\mathbb{F}_p$ denotes the finite field consisting of $p$ elements. Furthermore, we assume that any $\F$-algebra $A$ is finite-dimensional over $\F$ and any $A$-module is a left module (unless otherwise stated) and is also finite-dimensional over $\F$. Throughout, let $\NN_0$ and $\NN$ be the sets of non-negative and positive integers respectively. For any integers $a\leq b$, the set $\{a,a+1,\ldots,b\}$ consisting of consecutive integers is denoted as $[a,b]$.

Most of the background material in this section can be found in the textbooks such as \cite{Barot,Ben1,EH}.

\subsection{Basic algebras}\label{SS:basic algebra}  An algebra is basic if all its simple modules are 1-dimensional.
Descent algebras are basic, and therefore to study them, one can exploit tools from representation theory of algebras.
By a theorem of Gabriel,  a basic algebra over an algebraically closed field has a presentation by quiver and relations.
We recall the relevant input, since such a presentation of the descent algebras is essential for us.

\bigskip

A (finite) quiver $Q$ is a directed graph $Q=(Q_0,Q_1)$ where $Q_0$ is a finite set of vertices and $Q_1$ is a finite set of arrows. For each arrow $\gamma\in Q_1$, we write $s(\gamma)$ and $t(\gamma)$ for the starting and terminating vertices of $\gamma$ respectively. The quiver may have multiple loops or arrows. The underlying graph of $Q$ is the undirected graph obtained from $Q$ by forgetting the orientations of the arrows.

A path $\alpha$ of length $\ell$ in $Q$ is a sequence \[\alpha=\alpha_\ell\ldots\alpha_1\] where $\alpha_1,\cdots,\alpha_\ell$ are arrows and, for each $i\in [1,\ell-1]$, $s(\alpha_{i+1})=t(\alpha_i)$. In this case, we write $s(\alpha)=s(\alpha_1)$ and $t(\alpha)=t(\alpha_\ell)$.
Moreover, for each vertex $v\in Q_0$, we define the path of length zero $e_v$ where $s(e_v)=v=t(e_v)$. Let $\beta$ be another path in $Q$. We write $\beta\alpha$ for the composition of the paths if $t(\alpha)=s(\beta)$.  We use the  convention that the composition of the paths
is read from right to left.
In the case $v=s(\alpha)$, we have  $\alpha e_v=\alpha$, and similarly  $e_w\alpha=\alpha$ if $w=t(\alpha)$.

Let $Q$ be a quiver. The path algebra $\F Q$ is the $\F$-algebra with a formal basis the set of all paths in $Q$ where, for any paths $\alpha,\beta$, we define $\beta\cdot \alpha$ as $\beta\alpha$ if $s(\beta)=t(\alpha)$ and 0 otherwise.
In particular the paths $e_v$ of length zero give a set of orthogonal idempotents of $\F Q$, and  their sum $\sum_{v\in Q_0} e_v$ is the identity of the algebra.
These idempotents help to determine the composition multiplicities: if $V$ is an $\F Q$-module, and   $(V:S_i)$ is the  multiplicity of the simple $\F Q$-module $S_i$ as a composition factor of $V$, then $(V:S_i)$ is equal to $\dim_{\F}e_iV$.
 The (two-sided) ideal of $\F Q$ generated by all arrows of positive lengths is denoted as $\F Q^+$. More generally, for any $n\in\NN$, $(\F Q^+)^n$ denotes the ideal generated by all paths of lengths at least $n$.
 Notice that $\F Q$ is finite-dimensional if and only if $Q$ contains no oriented cycles.

 We have the following fundamental result by Gabriel.


\begin{thm}[\cite{Gabriel:1979a}]\label{T: path alg rel} Let $A$ be a basic $\F$-algebra. Then $A\cong \F Q_A/I$ where $I$ is an ideal such that $(\F Q_A^+)^n\subseteq I\subseteq (\F Q_A^+)^2$ for some integer $n\geq 2$.
\end{thm}

The quiver $Q_A$ in Theorem \ref{T: path alg rel} is the Ext quiver of  $A$, sometimes called Gabriel quiver. We recall its definition.
Its vertices are labelled by a full set of pairwise non-isomorphic simple $A$-modules $S_1, \ldots, S_r$ which is in bijection
with the set of paths  $e_i$ of length zero in $\F Q_A$, that is the $A$-module $Ae_i$  is the projective cover of the simple module
$S_i$.  The number of arrows $i\to j$ is defined to be the dimension of
${\rm Ext}_A^1(S_i, S_j)$,  recall that this space is  isomorphic to
\[\Hom_A(P_j,\rad(P_i))/\Hom_A(P_j,\rad^2(P_i))=e_j\rad(A)e_i/e_j\rad^2(A)e_i.\]
Here $\rad(A)$ is the   Jacobson radical of $A$, and for any $A$-module $V$, we write $\rad^n(V)=\rad^n(A)V$ for the $n$-th radical power, where   $n\in\NN$. Note that $\dim {\rm Ext}^1_A(S_i, S_j) =m \neq 0$ if and only if there is an indecomposable $A$-module of radical length two with socle the direct sum of
$m$ copies of  $S_j$ and top $S_i$. The Ext quiver of $A$ plays a crucial role in understanding the representation theory of $A$.

\subsection{Representation type}

Let $A$ be an $F$-algebra. It has finite representation type if there are only finite number of non-isomorphic indecomposable $A$-modules. It has tame representation type if it is not of finite type and, up to isomorphism, for each $d\in\NN$, almost all (except finitely many) indecomposable $A$-modules of dimension $d$ can be parametrised by a finite number of 1-parameter families.  The algebra  $A$ has wild representation type if, loosely speaking, its representations incorporate the representations of all finite-dimensional algebras, details can be found for example in \cite{Ben1}.
By a theorem of Drozd (see \cite{CB,Dr})  an algebra which is not of  finite type (that is, infinite type) is either tame or wild but not both.


We have the following famous result by Gabriel.

\begin{thm}[\cite{Gabriel:1972a}]\label{T: Gab} The path algebra $\F Q$ has finite type if and only if the underlying graph of $Q$ is a disjoint union of Dynkin diagrams of types $\mathbb{A}$, $\mathbb{D}$ or $\mathbb{E}$. In this case, up to isomorphism, the indecomposable $\F Q$-modules are in one-to-one correspondence with the positive roots.
\end{thm}

For the path algebras of tame type, we have the following classification.

\begin{thm}[\cite{Dlab/Ringel:1976a,D-F}]\label{T:tame} Suppose that $Q$ is connected without oriented cycles. Then $\F Q$ is tame if and only if the underlying graph of $Q$ is a (simply laced) extended Dynkin diagram.
\end{thm}

A variation of this gives a criterion for the representation type of an algebra in some special cases,  using the separated quiver of its Ext quiver. We recall the construction.
Suppose that $Q$ is a quiver with  vertices  $\{1, 2, \ldots, r\}$. The separated quiver $Q'$ of $Q$ is defined to be the quiver with   vertices $\{ 1, 2, \ldots, r, 1', 2', \ldots, r'\}$ and arrows
$i\to j'$ for any  arrow $i\to j$ in $Q$.
It was proved in \cite[X,  Theorem 2.4]{Auslander/Reiten/Smalo:1995a} that when $\rad^2(A)=0$, the algebra $A$ is stably equivalent to an algebra whose
Ext quiver is $Q'$ where $Q$ is the Ext quiver of $A$.  A general result of Krause \cite{Kr} implies that stably equivalent algebras have
the same representation type, and we can deduce the following.

\begin{thm}\label{T:separatedquiver} Let $A$ be an  $\F$-algebra and suppose that $\rad^2(A)=0$.
\begin{enumerate}[(i)]
\item {{\cite[X, Theorem 2.6]{Auslander/Reiten/Smalo:1995a}}} Then $A$ has finite type if and only if the separated quiver of $Q_A$ is a finite union of Dynkin diagrams.
\item {\cite{Dlab/Ringel:1976a,D-F}} Then $A$ is tame if and only if the separated quiver of $Q_A$ is a finite union of Dynkin diagrams and (at least one) extended Dynkin diagrams.
\end{enumerate}
\end{thm}

We end the subsection with the following two lemmas which we shall need in Section \ref{S: rep type}.

\begin{lem}\label{L: surj alg} Let $A$ and $B$ be $F$-algebras.
\begin{enumerate}[(i)]
  \item Suppose that there is a surjective algebra homomorphism $\phi:A\to B$. If $B$ has infinite (respectively, wild) type then $A$ has infinite (respectively, wild) type.
  \item {\cite[Proposition 2.2]{ESY21}} If there is an idempotent $e\in A$ such that $eAe$ is wild, then $A$ is also wild.
  \item {\cite[Proposition 1.2]{EHIS}} If $B=A/\soc(P)$ where $P$ is both a projective and injective $A$-module, then
$A$ and $B$ have the same representation type.
\end{enumerate}
\end{lem}

The first part of this lemma is elementary: $B$ is isomorphic
to $A/I$ for some ideal $I$ of $A$ and therefore viewing $B$-modules
as $A$-modules gives a functor from the category of $B$-modules to the category of $A$-modules which preserves
indecomposability and reflects isomorphisms. Moreover, the Ext quiver of $B$ is a subquiver of the Ext quiver of $A$ as follows.
First,
simple
$B$-modules are $A$-modules and remain simple. Therefore the vertices of the Ext quiver
of $B$ can be identified with vertices of the Ext quiver of $A$ induced by the surjection. Next, if there are $n$ arrows in the Ext quiver of $B$ from $i$ to $j$ then
there is  an indecomposable $B$ module of radical  length $2$ with socle the direct sum of $n$ copies of $S_j$ and top $S_i$ as we have seen in Subsection \ref{SS:basic algebra}. This is still an indecomposable $A$-module of radical length
$2$ and therefore gives rise to $n$ arrows in the Ext quiver of $A$ from $i$ to $j$ under the identification.  We have therefore:

\begin{lem} \label{L:quiver} Suppose that there is a surjective algebra
	homomorphism $\phi: A\to B$. Then the Ext quiver of $B$ is a subquiver
	of the Ext quiver of $A$.
\end{lem}

\bigskip
Part (iii) of Lemma \ref{L: surj alg} was used in \cite{EHIS}, in the context of radical embeddings, that is, if $A$ is a subalgebra of an algebra  $B$ then the inclusion $A \to B$ is
a radical embedding provided  $\rad(A)= \rad(B)$.  In this context, if the algebra $B$ has finite type then, roughly speaking, the main result shows that $A$ is `not too far away'.
The following shows that sometimes also $A$ has finite type, and we will apply this later.

\begin{lem}\label{L:radicalemb} Suppose that $A$ is a subalgebra of $B$ such that  $\rad(A)=\rad(B)$.  Assume that
  $B \cong A\oplus (B/A)$ as a left $A$-modules. If $B$ has finite type then $A$ also has finite type.
\end{lem}
\begin{proof} It is enough to show that $A$ has finitely many indecomposable modules which are not simple. Take such a module $X$. Then
$X^-:= {\rm Hom}_A(B, X)$ is a $B$-module.
Since $B= A\oplus (B/A)$  we have, as $A$-modules,
$$X^- \cong {\rm Hom}_A(A, X)\oplus {\rm Hom}_A(B/A, X) \cong X \oplus W$$
where $W$ is semisimple (see  \cite[Lemma 2.3]{EHIS}).

By the Krull-Schmidt theorem, there is an indecomposable $B$-module $Y$, which is a summand of $X^-$ and the restriction $Y_A$ of $Y$ to $A$ contains $X$ as a direct summand. In fact, $Y_A= X\oplus W'$ with $W'$ semisimple.  Since $X$ is not simple, the module $Y$ is unique and its restriction to $A$ has a unique non-simple indecomposable direct summand. So we have a map from indecomposable non-simple $A$-modules to indecomposable $B$-modules given by  $X \mapsto  Y$. By the construction, up to isomorphism, the map must be injective as follows. If $X_1\mapsto Y$ and $X_2\mapsto Y$, then the restriction of $Y$ to $A$ contains both indecomposable non-simple summands $X_1$ and $X_2$. But there is a unique  non-simple summand and hence $X_1\cong X_2$. Since the set of such non-isomorphic indecomposable $B$-modules $Y$ is finite, the set of non-isomorphic indecomposable non-simple $A$-modules $X$ is finite.
\end{proof}

\subsection{Descent algebras of type $\A$}\label{SS: descent algebra}  In this section, the field $\F$ is  arbitrary.
We describe the descent algebra of the symmetric group $\sym{n}$.  The composition of  permutations is read from left to right.

A composition $q$ of $n$ is a finite sequence $(q_1,\ldots,q_k)$ in $\NN$ such that $\sum_{i=1}^kq_i=n$. In this case, we write $q\vDash n$\index{$\vDash$} and $\ell(q)=k$.
The Young subgroup associated to $q$ is the group
$\sym{q}$, which is
the direct
product
$\sym{q_1}\times \sym{q_2}\times \cdots \times \sym{q_k}$ in $\sym{n}$. More precisely, it consists of the permutations whose support sets are
\[[1, q_1], [q_1+1, q_1+q_2], \ldots, [q_1 + \cdots + q_{k-1}, n].\]
Each right coset of $\sym{q}$ in $\sym{n}$ contains a unique permutation of minimal length. Define $\Xi^{q}$ to be the sum in the integral group ring $\Z \sym{n}$ of all these
minimal coset representatives of $\sym{q}$ in $\sym{n}$. For example, $\Xi^{(n)}$ is the identity of $\sym{n}$, and $\Xi^{(1, 1, \ldots, 1)}$ is the sum of all permutations in $\sym{n}$. The result by Solomon \cite{Solomon:1976a} applied to the $\A$ type case asserts that  the $\Z$-linear span
$\Des{n}{\Z}$ of the elements $\Xi^{q}$ as $q$ varies through the compositions
of $n$, is a subring of $\Z \sym{n}$ of rank $2^{n-1}$. The subalgebra $\Des{n}{\Z}$ is called a descent algebra (of type $\mathbb{A}$). 
There is an explicit formula for the multiplication constants for the $\Z$-basis $\{\Xi^q:q\vDash n\}$ of $\Des{n}{\Z}$, which is also
in connection with the Young permutation characters. We will describe these in detail.
Furthermore, we may replace $\Z$ as a coefficient ring by any commutative ring $\cO$
with identity, and get the algebra $\Des{n}{\cO}$, which is $\cO$-free  of rank $2^{n-1}$.

\bigskip

\subsection{The multiplication and connection with Young permutation modules}

In order to describe the multiplication constants and consequences, we need some notation.
The composition $q$ is a partition if $q_1\geq q_2\geq \cdots\geq q_k$ and we write $q\vdash n$.  We denote the set of all compositions and partitions of $n$ by $\C(n)$ and $\P(n)$ respectively\index{$\C(n)$}\index{$\P(n)$}. A partition $\lambda$ is $p$-regular if, for each $i\in\NN$, the number of occurrence of $i$ as a part of $\lambda$ is not more than $p-1$. The set of $p$-regular partitions of $n$ is denoted as $\pReg(n)$. Obviously, $\Lambda^+_\infty(n)=\P(n)$.
To simplify the notation, if no confusion, we remove all parentheses and commas for the notation of composition. For example, $(2,1,1)$ is replaced with $21^2$.

Let $q,r\in\C(n)$. If the parts of $q$ can be rearranged to $r$ then we write $q\approx r$. Clearly, this is an equivalence relation and the equivalence classes are represented by $\P(n)$. As such, we write $\lambda(q)$ for the partition such that $q\approx \lambda(q)$. The composition $r$ is a (strong) refinement of $q$ if there are integers $0=i_0<i_1<\cdots<i_k$ where $k=\ell(q)$ such that, for each $j\in [1,k]$, \[r^{(j)}:=(r_{i_{j-1}+1},r_{i_{j-1}+2},\ldots,r_{i_j})\vDash q_j,\] and we denote it as $r\sref q$. On the other hand, the composition $r$ is a weak refinement of $q$ if there is a rearrangement of $r$ which is a refinement of $q$, i.e., $r\approx s\sref q$ for some $s$ and we denote this as $r\wref q$.

For an $(m\times n)$-matrix $A$, we denote the $i$th row and $j$th column of $A$ by $r_i(A)$ and $c_j(A)$ respectively, i.e.,
\begin{align*}
  r_i(A)&=(A_{i1},\ldots,A_{in}),\index{$r_i(A)$}\\
  c_j(A)&=(A_{1j},\ldots,A_{mj})\index{$c_j(A)$}.
\end{align*} Let $q,r,s\in \C(n)$ and $N_{r,q}^s$ \index{$N_{r,q}^s$}be the set consisting of all the $(\ell(r)\times \ell(q))$-matrices $A$ with entries in $\NN_0$ such that
\begin{enumerate}[(a)]
\item for each $i\in [1,\ell(r)]$, $\comp{r_i(A)}\vDash r_i$,
\item for each $j\in [1,\ell(q)]$, $\comp{c_j(A)}\vDash q_j$, and
\item $s=\comp{r_1(A)}\cont\cdots\cont \comp{r_{\ell(q)}(A)}$ where $\cont$ denotes concatenation,
\end{enumerate} moreover, $\comp{\delta}$ \index{$\comp{\delta}$}denotes the composition obtained from a sequence $\delta$ in $\NN_0$ by deleting all its zero entries.

The following theorem gives the multiplication constants for the type $\A$ case.

\begin{thm}[{\cite[Proposition 1.1]{Garsia/Reutenauer:1989a}}]\label{T: GR 1.1} For $q,r\in\C(n)$, we have \[\Xi^r\Xi^q=\sum_{s\in\C(n)}|N^s_{r,q}|\Xi^s.\] 
\end{thm}

We now describe the connection of Theorem \ref{T: GR 1.1} with permutation modules.

Recall that, for any element $x$ of finite order of a group $G$, we have a unique decomposition $x=yz=zy$ for some $y,z\in G$ such that the orders of $y$ and $z$ are a power of $p$ and prime to $p$ respectively. We call $z$ the $p'$-part of $x$. The conjugacy class of $\sym{n}$ labelled by $\lambda\in\P(n)$ is denoted by $\ccl{\lambda}$.  Any partitions $\lambda,\mu\in\P(n)$ are $p$-equivalent if the $p'$-parts of any $\sigma\in\ccl{\lambda}$ and $\tau\in\ccl{\mu}$ are conjugate in $\sym{n}$. In this case, we write $\lambda\sim_p\mu$  for the equivalence relation. Notice that $\lambda\sim_\infty\mu$ if and only if $\lambda=\mu$. The $p$-equivalence classes of $\sym{n}$ are represented by $\pReg(n)$. For each $\lambda\in\pReg(n)$, the corresponding $p$-conjugacy class $\ccl{\lambda,p}$ is the union of conjugacy classes of the form $\ccl{\mu}$ such that $\mu\sim_p\lambda$. In this case, notice that, up to rearrangement of parts, $\lambda$ can be obtained from $\mu$ by successively adding some $p$ equal parts and therefore $\mu\wref \lambda$.

Let $q\vDash n$. The Young character $\varphi^{q,\Z}$ (or simply $\varphi^q$) is defined as the character of the permutation module $M^q_\Z=\Ind^{\sym{n}}_{\sym{q}}\Z$ where $\Z$ is the considered as the trivial $\Z\sym{q}$-module. In fact, for each $\mu\in\P(n)$, $\varphi^q(\mu)$ is the number of right cosets of $\sym{q}$ in $\sym{n}$ fixed by a permutation with cycle type $\mu$, where we have identified $\mu$ with the conjugacy class $\ccl{\mu}$ of $\sym{n}$. Therefore, $\varphi^q=\varphi^{q'}$ if $q\approx q'$. Also, we denote $\varphi^{q,\cO}$ for the $\cO$-valued Young character, that is, for any $\mu\in\P(n)$, \[\varphi^{q,\cO}(\mu)=\varphi^q(\mu)\cdot 1_{\cO}\in \cO.\]  We have the following lemma.

\begin{lem}\label{L: phi nonzero} Let $q\vDash n$ and $\mu\in\P(n)$.
\begin{enumerate}[(i)]
  \item If $\varphi^q(\mu)\neq 0$ then $\mu\wref q$.
  \item {\cite[Lemma 3.1]{Lim:2023a}} If $\mu\in\P_p(n)$, we have $\varphi^{\mu,\F}(\mu)\neq 0$.
\end{enumerate}
\end{lem}

Let $\Ccl{n}{\F}$ be the $\F$-linear span of the Young characters $\varphi^{q,\F}$'s.   The following well-known identity \[\varphi^q \varphi^r=\sum_{s\in\C(n)} |N^s_{q,r}|\varphi^s,\] which comes from the Mackey formula, together with Theorem \ref{T: GR 1.1}  gives rise to the next theorem. 

\begin{thm}[{\cite{Atkinson/vanWilligenburg:1997a,Solomon:1976a}}]\label{T: Sol epi}  The $\F$-linear map \[\sol_{n,\F}:\Des{n}{\F}\to \Ccl{n}{\F}\] sending $\Xi^q$ to $\varphi^{q,\F}$ is a surjective $\F$-algebra homomorphism. Furthermore, $\ker(\sol_{n,\F})=\rad(\Des{n}{\F})$ and  is the $\F$-span of the set consisting of $\Xi^q$ such that $\lambda(q)\not\in\pReg(n)$, together with the $\Xi^q-\Xi^r$ such that $q\approx r$ with $q\neq r$. In particular, $\Des{n}{\F}$ is a basic algebra.
\end{thm}

In the case when the context is clear, we simply write $\sol$ for $\sol_{n,\F}$.

Recall that the nilpotency index of an algebra $A$ is the smallest positive integer $m$ such that $\rad^m(A)= 0$. We have the following result on the nilpotency index of $\Des{n}{\F}$ which is independent of $p$.

\begin{thm}[{\cite[Corollary 3.5]{Atkinson:1992a}, \cite[Theorem 3]{Atkinson/vanWilligenburg:1997a}}]\label{T:radlen} Suppose that $n\geq 3$. The nilpotency index of $\Des{n}{\F}$ is $n-1$.
\end{thm}

\subsection{Simple  modules and Cartan matrices}

We shall now describe the simple modules for $\Des{n}{\F}$ as in \cite{Atkinson/vanWilligenburg:1997a}. The simple $\Des{n}{\F}$-modules are parametrised by $\pReg(n)$. For each $\lambda\in \P_p(n)$, let $M_{\lambda,\F}$ be the one-dimensional vector space over $\F$ such that, for $v\in M_{\lambda,\F}$ and $\alpha\in \Des{n}{\F}$, we have \[\alpha\cdot v=\sol(\alpha)(\lambda)v.\] In particular, $\Xi^q\cdot v=\varphi^{q,\F}(\lambda)v$. Since the Young characters take values in $\Z$, we may lift the simple modules. More generally, for each $\mu\in\P(n)$, let $M_{\mu,\cO}$ be the $\cO$-free module of rank one  such that, for $v\in M_{\mu,\cO}$, \[\Xi^q\cdot v=\varphi^{q,\cO}(\mu)v.\] Notice that we have $M_{\lambda,\F}\cong M_{\mu,\Z}\otimes_{\Z} \F$ if and only if $\lambda\sim_p\mu$.

Recall that, in general,  the
Cartan numbers of an algebra $A$ are the integers $C_{ij}:=(P_i:S_j)$ which are the composition multiplicities of the projective indecomposable $A$-modules.  The matrix $C=(C_{ij})_{i,j\in[1,r]}$ is called the Cartan matrix of $A$ where $r$ is the total number of non-isomorphic simple (or projective indecomposable) $A$-modules.
For an algebra which is defined over $\Z$, the Cartan matrices over $\Q$ and over $\mathbb{F}_p$ can be related via  the decomposition matrix. The entries of the decomposition matrix are the composition multiplicities of the $A$-modules obtained by the $p$-modular reduction (see Subsection \ref{SS:reductionp}) of the simple modules defined over $\Q$.

We describe this for the descent algebras in type $\A$.
We fix the total order $\leq$ on $\P(n)$ defined by the lexicographic order which refines the partial order on $\P(n)$ defined by the weak refinement $\wref$. This induces a total order on the subset $\pReg(n)$ of $\P(n)$. Let $D$ be the matrix with the rows and columns labelled by $\P(n)$ and $\pReg(n)$ respectively such that, for any $\lambda\in\P(n)$ and $\mu\in\pReg(n)$, $d_{\lambda,\mu}=1$ if $\lambda\sim_p\mu$ and $d_{\lambda,\mu}=0$ otherwise. The matrix $D$ is  the decomposition matrix of the descent algebra (of degree $n$ with respect to the prime $p$).

Since the simple modules for the descent algebras are defined over $\Z$, their Cartan matrices depend only on the characteristic of the field. When $p=\infty$, the Cartan matrix $C$ has been formulised by Blessenohl--Laue \cite{Blessenohl/Laue:2002a}. Since we do not need the full description of the formula, we refer the reader to \cite[Corollary 2.1]{Blessenohl/Laue:2002a}. By the description of $C_{qr}$, if $C_{qr}\neq 0$, then $q$ is a weak refinement of $r$ and hence, by the choice of our total order on $\P(n)$, we have $q\leq r$. That is, $C$ is upper triangular. When $p$ is finite, the Cartan matrix $\wt{C}$ can be obtained using the result of Atkinson--Pfeiffer--van Willigenburg \cite{Atkinson/Pfeiffer/vanWilligenburg:2002a}. Their result holds for the descent algebras of all finite Coxeter groups but, again, we only need it for the type $\A$ case.




\begin{thm}[{\cite[Theorem 8]{Atkinson/Pfeiffer/vanWilligenburg:2002a}}]\label{T: APW} Let $C$ be the Cartan matrix of $\Des{n}{\Q}$ and $D$ be the decomposition matrix. Then the Cartan matrix of $\Des{n}{\mathbb{F}_p}$ is $\wt{C}=D^\top CD$.
\end{thm}

We give  an example which we shall need later  but leave the details to the reader.

\begin{eg}\label{Eg: D6 Cartanp} When $n=6$, the Cartan matrices in the cases when $p\geq 7$ and $p=5$ are $C$ and $\wt C$ respectively given as below:
with respect to the lexicographic order
\begin{align*}&1^6<
 21^4<
2^21^2<
2^3<
3 1^3 <
 3 2 1<3^2 <
4 1^2<
4 2 <
 51 <
 6 ,
\end{align*} we have
\begin{align}\label{Eq:C60}
C&=\begin{pmatrix}
1&0&0&0&0&0&0&0&0&0&0\\
0&1&0&0&1&0&0&1&0&1&1\\
0&0&1&0&0&1&1&0&1&1&1\\
0&0&0&1&0&0&0&0&0&0&0\\
0&0&0&0&1&0&0&1&0&1&1\\
0&0&0&0&0&1&1&0&1&1&2\\
0&0&0&0&0&0&1&0&0&0&0\\
0&0&0&0&0&0&0&1&0&1&1\\
0&0&0&0&0&0&0&0&1&0&1\\
0&0&0&0&0&0&0&0&0&1&1\\
0&0&0&0&0&0&0&0&0&0&1\\
\end{pmatrix},
& \wt{C}&=\begin{pmatrix}
1&0&0&1&0&0&1&0&1&1\\
0&1&0&0&1&1&0&1&1&1\\
0&0&1&0&0&0&0&0&0&0\\
0&0&0&1&0&0&1&0&1&1\\
0&0&0&0&1&1&0&1&1&2\\
0&0&0&0&0&1&0&0&0&0\\
0&0&0&0&0&0&1&0&1&1\\
0&0&0&0&0&0&0&1&0&1\\
0&0&0&0&0&0&0&0&2&1\\
0&0&0&0&0&0&0&0&0&1\\
\end{pmatrix}.
\end{align}
\end{eg}

In the $p=\infty$ case,  different sets of explicit primitive orthogonal idempotents of $\Des{n}{\F}$ have been obtained. For example, those given by Garsia--Reutenauer \cite{Garsia/Reutenauer:1989a}. In the $p<\infty$ case, Erdmann--Schocker \cite[Corollary 6]{ES} showed that there exists a complete set of primitive orthogonal idempotents \[\{e_{\lambda,\F}:\lambda\in\pReg(n)\}\] for $\Des{n}{\F}$ such that $\sol(e_{\lambda,\F})=\Char_{\lambda,\F}$ where $\Char_{\lambda,\F}$ is the characteristic function with respect to the $p$-equivalence class $\ccl{\lambda,p}$. In the recent paper \cite[\S3]{Lim:2023a}, the second author gave a construction for such $e_{\lambda,\F}$'s. We shall not need the complete description of the idempotents but, theoretically, they can be used to compute the Ext quivers which we shall need later, especially for the small cases.

\subsection{The Ext quiver of $\Des{n}{F}$}
As before, the Ext quiver of $\Des{n}{\F}$ depends only on $p$ but not on the particular field and we denote it by $Q_{n,p}$.  Since the simple modules are the
$M_{\lambda, \F}$'s, one for each $\lambda \in \pReg(n)$, we  label the vertices of $Q_{n, p}$ by the $p$-regular partitions of $n$.
  Let $e_{\lambda,\F}$ be the idempotent corresponds to the simple module $M_{\lambda,\F}$. We write $P_{\lambda,\F}=\Des{n}{\F}e_{\lambda,\F}$ for the projective cover of $M_{\lambda,\F}$.  Since $\Des{n}{\F}$ is basic, we have that \[\Des{n}{\F}\cong \F Q_{n,p}/I\] for some ideal $I$ of $\F Q_{n,p}$ as in Theorem \ref{T: path alg rel}. The Ext quivers of $\Des{n}{\F}$ in the case of $p>n$ (including $p=\infty$) have been determined by Schocker and Saliola as follows.

\begin{thm}[{\cite[Theorem 8.1]{Saliola:2008a},\cite[Theorem 5.1]{Schocker:2004a}}]\label{T:OrdExtquiver} Suppose that $p>n$  and $\lambda,\mu\in\pReg(n)$. The Ext quiver of $\Des{n}{\F}$ has an arrow $\lambda\to \mu$ if and only if $\mu$ is obtained from $\lambda$ by adding two distinct parts of $\lambda$. That is,  there exist positive integers $a_1,\ldots,a_r,x,y$ with $x\neq y$ such that $\lambda \approx (a_1, \ldots, a_r, x, y)$ and $\mu \approx (a_1, \ldots, a_r, x+y)$.
\end{thm}

In \cite{Schocker:2004a}, using just the description of the Ext quiver, Schocker obtained the representation type of $\Des{n}{\F}$ when $p=\infty$.

\begin{thm}[{\cite[\S5.1.2]{Schocker:2004a}}] When $p=\infty$, the descent algebra $\Des{n}{\F}$ has finite type if $n\leq 5$,  and wild type otherwise.
\end{thm}

\subsection{$\Des{n}{\F}$ for small $n$}  Later, to identify the representation type of descent algebras in the modular case, we will need
some details on $\Des{n}{\F}$ for some small $n$.  It will turn out that knowing the Ext quiver together with some general theory will be
sufficient for us.

\begin{cor}\label{C:nleq5} Let $n\in [2,5]$ and $p>n$ (including $p=\infty$). Then $\Des{n}{\F}$ is isomorphic to the path algebra $\F Q_{n,p}$.
\end{cor}
\begin{proof} For each of such $n$ and $p$, using Theorem \ref{T:OrdExtquiver}, it is plain to check that the path algebra $\F Q_{n,p}$ has dimension $2^{n-1}$ which is the same as the dimension of $\Des{n}{F}$. By Theorem \ref{T: path alg rel}, we have $\Des{n}{F}\cong \F Q_{n,p}$.
\end{proof}

\begin{cor}\label{C:n6} Suppose that $p\geq 7$. We have $\Des{6}{\F}\cong \F Q_{6,p}/I$ where $I$ is the $1$-dimensional ideal of $\F Q_{6,p}$ spanned by an element $\omega$ satisfying $\omega=e_{6,\F}\omega e_{2^21^2,\F}$.
\end{cor}
\begin{proof} By Theorem \ref{T:OrdExtquiver}, the quiver $Q=Q_{6,p}$ is given as follows:
\begin{equation}\label{Dg:Q60}
\begin{tikzcd}
    6&51\ar[l]&41^2\ar[l]&31^3\ar[l]&21^4\ar[l]\\ 42\ar[u]&321\ar[u]\ar[l]\ar[r]&3^2&2^3&1^6\\ &2^21^2\ar[u]
  \end{tikzcd}
\end{equation} The algebra $\F Q$ has dimension $33$ and hence
$\Des{6}{\F} \cong \F Q/I$ where
$I$ is a 1-dimensional ideal of $\F Q$.
The ideal $I$ must then be spanned by an element $\omega = e_{\lambda,\F}\omega e_{\mu,\F}$ for some $\lambda, \mu$. We claim that $\lambda = (6)$ and $\mu = (2^2,1^2)$.

Since $I$ is an ideal, the vertices labelled by $\lambda$ and $\mu$ must be a sink and a source in $Q$ respectively. Moreover there must be a path from $\mu$ to
$\lambda$. Clearly, we have neither $\lambda = (1^6)=\mu$ nor $\lambda=(2^3)=\mu$, if not, $\Des{6}{\F}$ would not have a simple module labelled by $(1^6)$ or $(2^3)$. So $\lambda$ is either $(6)$ or $(3^2)$. Similarly, $\mu$ is either $(2^2,1^2)$ or $(2,1^4)$. We claim that $\mu=(2^2,1^2)$. Suppose on the contrary that $\mu = (2,1^4)$. We must have $\lambda=(6)$ and $\omega$ is a non-zero multiple of the unique path of $\F Q$ of length $4$. So $\rad^4(\Des{6}{\F})=0$. This contradicts Theorem \ref{T:radlen}. Hence $\mu = (2^2,1^2)$. By examining the Cartan matrix $C$ in (\ref{Eq:C60}), we see that $C_{2^21^2,3^2}=1\neq 0$ and hence $\lambda\neq (3^2)$. As such, $\lambda=(6)$.
\end{proof}

\subsection{Reduction modulo $p$}\label{SS:reductionp} In Section \ref{S: rep type}, we wish to obtain partial information about $\Des{n}{\F}$ when $p<\infty$, using representations from characteristic
zero.   The canonical method for this  base change involves a $p$-modular system.  We take for $\cO$ a local principal ideal domain with the maximal ideal $(\pi)$, and we let $K$ be the field of fractions of $\cO$, and the field of characteristic $p$ is   $\FF= \cO/(\pi)$. For the descent algebras,
 we may take $K=\mathbb{Q}$, $\cO=\mathbb{Z}_{(p)}=\{ x\in \bQ\mid  \nu_p(x) \geq 0\}\cup \{ 0\}$ and $\FF=\mathbb{F}_p$.

Let $B$  be an $\cO$-algebra which are finitely generated and free over $\cO$, and let $B$-\textsf{mod} be the category consisting of finitely generated left $B$-modules. We write $B_K:=B\otimes_\cO K$, $B_\FF=B\otimes_\cO \FF$ and, for any object $V$ of $B$-\textsf{mod} that is $\cO$-free, we write $V_K:=V\otimes_\cO K$ and $V_\FF:=V\otimes_\cO \FF$ for the $B_K$- and $B_\FF$-modules respectively. Any finite-dimensional $B_K$-module $W$ has an $\cO$-form, that is,
there is a $B$-module $V$ which is $\cO$-free and   $V_K\cong W$. Such an $\cO$-form  $V$ (of $W$) is not unique but, by the general theory, the composition multiplicities of $V_\FF$ do not depend on the choice of $V$.

\begin{lem}[{\cite[Lemma 1.5.2]{Cline/Parshall/Scott:1996a}}]\label{L:CPS} Suppose that $V$ and $T$ are objects in $B$-\textsf{mod} which are free over $\cO$.
\begin{enumerate}[(i)]
\item The canonical  map $\Hom_{B}(V, T)_K \to \Hom_{B_K}(V_K, T_K)$ is an isomorphism.
\item The canonical map $\Hom_{B}(V, T)_\FF \to \Hom_{B_\FF}(V_\FF, T_\FF)$ is injective.
\item If $\Ext^n_{B_\FF}(V_\FF, T_\FF)=0$, then $\Ext^n_{B}(V, T)=0$.
\item If $\Ext^n_{B}(V, T) \neq 0$, then
	$\Ext^n_{B_\FF}(V_\FF, T_\FF)\neq 0$ via the natural base change map.
\item If $\Ext^1_{B}(V, T)=0$, then the map $\Hom_{B}(V, T)\to \Hom_{B_\FF}(V_\FF, T_\FF)$ is surjective.
\end{enumerate}
\end{lem}


Lemma \ref{L:CPS}(i) implies that, for any objects $V$ and $T$ in $B$-\textsf{mod} that are $\cO$-free,  we have
\begin{equation}\label{Eq:Ext1}
\Ext^1_{B}(V, T)_K \cong \Ext^1_{B_K}(V_K, T_K).
\end{equation} Applying Lemma \ref{L:CPS} to the descent algebra case, we obtain the following corollary.

\begin{cor}\label{C:Ext1Descent} Let $\delta,\gamma\in\P(n)$ and $\lambda,\mu\in \pReg(n)$ such that $\delta\sim_p\lambda$ and $\gamma\sim_p\mu$. If $\Ext^1_{\Des{n}{K}}(M_{\delta, K}, M_{\gamma, K})\neq 0$, then
$\Ext^1_{\Des{n}{\FF}}(M_{\lambda, \FF}, M_{\mu, \FF})\neq 0$. In other words,  if there is an arrow $\delta\to \gamma$ in $Q_{n,\infty}$, then there is an arrow $\lambda\to \mu$ in $Q_{n,p}$.
\end{cor}
\begin{proof} Notice that $\Des{n}{\cO}$ is $\cO$-free and $M_{\delta,\cO}$ is $\cO$-free of rank $1$ for any $\delta\in\P(n)$. By (\ref{Eq:Ext1}) and our assumption, $\Ext^1_{\Des{n}{\cO}}(M_{\delta, \cO}, M_{\gamma, \cO})\neq 0$. By Lemma \ref{L:CPS}(iv) with $n=1$, we have $\Ext^1_{\Des{n}{\FF}}(M_{\lambda, \FF}, M_{\mu, \FF})\neq 0$ as $M_{\lambda,\FF}\cong M_{\delta, \cO}\otimes_\cO \FF$  and $M_{\mu,\FF}\cong M_{\gamma, \cO}\otimes_\cO \FF$.
\end{proof}

Since the $\cO$-forms of the simple $\Des{n}{K}$-modules have rank one, we have seen that the reduction modulo $p$ of them is uniquely determined by $M_{\lambda,\FF}\cong M_{\mu,\cO}\otimes_\cO \FF$ if $\mu\sim_p \lambda\in\pReg(n)$. On the other hand, let $\mu\in\P(n)$ and consider the projective module $P_{\mu,K}=\Des{n}{K}e_{\mu,K}$. Let $a\in\NN_0$ be such that $e':=p^ae_{\mu,K}\in\Des{n}{\cO}$ and $P_{\mu,\cO}=\Des{n}{\cO}e'$. Since $P_{\mu,\cO}$ is an $\cO$-submodule of $\Des{n}{\cO}$ and $\cO$ is a principal ideal domain, the module $P_{\mu,\cO}$ is $\cO$-free and hence is an $\cO$-form of $P_{\mu,K}$.

We now make an observation which is crucial for the proof of the next lemma. When $n\geq 3$ and $p>n$, $Q_{n,p}$ is acyclic with the unique longest path of length $n-2$:
\begin{align}\label{Eq:longestpath}
\beta: 21^{n-2}\to 31^{n-3}\to\cdots \to n.
\end{align} Since the nilpotency index of $\Des{n}{\F}$ is $n-1$ by Theorem \ref{T:radlen}, using \cite[proof of Corollary 3.5]{Atkinson:1992a}, we see that $\rad^{n-2}(\Des{n}{\F})$ is one-dimensional $\F$-spanned by $w^{n-2}$ where
\begin{align}\label{Eq:w}
w:=\Xi^{(n-1,1)}-\Xi^{(1,n-1)}.
\end{align}

\begin{lem}\label{L:projmodp} Let $\mu\in\P(n)$ and fix an $\cO$-form $P_{\mu,\cO}$ of $P_{\mu,K}$. Furthermore, let $T:=P_{\mu,\cO}\otimes_\cO \FF$ and $\lambda\in\pReg(n)$ such that $\lambda\sim_p\mu$, so that  $M_{\lambda,\FF}\cong M_{\mu,\cO}\otimes_\cO \FF$.
\begin{enumerate}[(i)]
\item The module $T$ is cyclic.
\item If $(T:M_{\lambda,\FF})=1$, then $T$ is indecomposable.
\item If $n\geq 3$ and $\mu=(2,1^{n-2})$, then $T$ is uniserial with the composition factors $M_{\lambda^{(i)},\FF}$ from top to bottom as $i$ runs through $0,\ldots,n-2$ where
$\lambda^{(i)}\sim_p (2+i,1^{n-2-i})$.
\end{enumerate}
\end{lem}
\begin{proof} For part (i), since $P_{\mu,\cO}$ is $\cO$-free, it has an $\cO$-basis $\{\Xi^qe':q\in \Psi\}$ for some subset $\Psi$ of $\C(n)$. Therefore $T$ is cyclic generated by $f:=e'\otimes \FF\in \Des{n}{\FF}$. For part (ii), consider a module homomorphism $\varphi:T\to T$. Since $T$ is cyclic, $\varphi$ is determined by the image of the generator $m\in T$. This must satisfy $e\varphi(m) = \varphi(em)$ for any idempotent $e\in \Des{n}{\FF}$. Taking $e=e_{\lambda}$, by the assumption, $\varphi(m)$ is a scalar multiple of $m$ and the endomorphism algebra is isomorphic to $\FF$. Thus $T$ is indecomposable. By Theorem \ref{T:radlen}, the radical length of $\Des{n}{\FF}$ is $n-1$ and the unique path $\beta$ of length $n-1$ in $Q_{n,\infty}$ is given as in (\ref{Eq:longestpath}). Since $\Des{n}{K}\cong KQ_{n,\infty}/I$, we see that $\beta$ must survive when $I$ is factored out, and also that it lives in $P_{21^{n-2},K}$ via the isomorphism. Let $w\in\rad(\Des{n}{\cO})$ be given as in (\ref{Eq:w}). The $\cO$-form $P_{21^{n-2},\cO}$ has $\cO$-basis $\{w^je': j\in [0,n-2]\}$. For each $i\in [0,n-1]$, let $N_i$ be the submodule of $P_{21^{n-2},\cO}$ $\cO$-spanned by $\{w^je': j\in [i,n-2]\}$. Then $T$ is uniserial with submodules \[0=N_{n-1}\otimes_\cO \F\subseteq N_{n-2}\otimes_\cO \F\subseteq \cdots\subseteq N_0\otimes_\cO \F=T\] and, furthermore, $N_i\otimes_\cO F/N_{i+1}\otimes_\cO F\cong M_{\lambda^{(i)},\F}$ for each $i\in [0,n-2]$.
\end{proof}

\section{The Bergeron-Garsia-Reutenauer map}\label{S:BGRmap}

Recall that $\cO$ is a commutative ring $\cO$ with 1. Let $n\geq s\geq 1$. For each $q\vDash n$ and $i\in [1,\ell(q)]$, if $q_i\geq s$, let $q^{(i)}$ be the composition of $n-s$ obtained from $q$ by subtracting the $i$th component $q_i$ of $q$ by $s$ and then removing the 0 at the $i$th position if $q_i=s$. Define an $\cO$-linear map $\Delta_s:\Des{n}{\cO}\to \Des{n-s}{\cO}$ by \[\Delta_s(\Xi^q)=\sum_{q_i\geq s} \Xi^{q^{(i)}}.\] For example,  $\Delta_1(\Xi^{(2,1)})=\Xi^{(2)}+\Xi^{(1,1)}$ and $\Delta_2(\Xi^{(2,1)})=\Xi^{(1)}$. In \cite{Bergeron/Garsia/Reutenauer:1992a}, Bergeron-Garsia-Reutenauer showed that, over $\cO=\Q$, $\Delta_s$ is a surjective algebra homomorphism.

\begin{thm}\label{T: BGR hom} Over  an arbitrary commutative ring $\cO$ with $1$, the map $\Delta_s:\Des{n}{\cO}\to\Des{n-s}{\cO}$ is a surjective algebra homomorphism.
\end{thm}
\begin{proof} The proof of \cite[Theorem 1.1]{Bergeron/Garsia/Reutenauer:1992a} works over $\cO$. So $\Delta_s$ is an algebra homomorphism. For the surjectivity of $\Delta_s$, we argue by induction. We totally order $\P(n-s)$ by the lexicographic order $>$ as before. Since partitions are equivalence classes of compositions with respect to the equivalence relation $\approx$, we have a partial order on the compositions of $n-s$. Clearly $\Delta_s(\Xi^{(n)})=\Xi^{(n-s)}$. Let $q\vDash n-s$. By induction hypothesis, assume that $\Xi^r\in\im(\Delta_s)$ whenever $\lambda(r)>\lambda(q)$. Let $b$ be the first index such that $q_b\geq q_j$ for all $j$. Let $q'$ be the composition of $n$ obtained from $q$ by adding $s$ to the $b$th component of $q$ so that $q'_b=q_b+s$. By definition, $\Delta_s(\Xi^{q'})=\Xi^q+\epsilon$ where $\epsilon$ is either 0 or a sum of some $\Xi^r$'s such that $r$ has a component $q_b+s$. For such $\Xi^r$, we have $\lambda(r)>\lambda(q)$. As such, $\Xi^q=\Delta_s(\Xi^{q'})-\epsilon\in\im(\Delta_s)$.
\end{proof}

By Theorem \ref{T: BGR hom}, we have a functor
\[\Fun _s:\text{$\Des{n-s}{\cO}$-\textsf{mod}}\to \text{$\Des{n}{\cO}$-\textsf{mod}}\] induced by the surjective algebra homomorphism $\Delta_s$, that is, for any $\Des{n-s}{\cO}$-module $V$, we have $\Fun _s(V)=V$ as $\cO$-module and, for $q\in \C(n)$ and $v\in V$, we have \[\Xi^q \cdot v=\Delta_s(\Xi^q)v.\] Since $\Delta_s$ is surjective, if $V,W$ are $\Des{n-s}{\cO}$-modules such that $\Fun_s(V)\cong \Fun_s(W)$, then $V\cong W$.

Using Theorem \ref{T: BGR hom} and Lemma \ref{L: surj alg}(i), we get the following corollary.

\begin{cor}\label{C: BGR hom} If $\Des{n-s}{\F}$ has infinite (respectively, wild)  type then $\Des{n}{\F}$ has infinite (respectively, wild)  type.
\end{cor}

\section{Pullback along $\Delta_s$}\label{S:pullback}

In this section, we study the functor $\Fun_s$ given in the previous section more closely.  Given $\mu\in\pReg(n-s)$, Theorem \ref{T: pullback simple} identifies the partition $\lambda\in\pReg(n)$ such that $M_{\lambda,\F}\cong \Fun _s(M_{\mu,\F})$. The construction of the partition $\lambda$ is given in Definition \ref{D:pullbackpar}. As an application, we prove that the quiver $Q_{n-s,p}$ is a subquiver of $Q_{n,p}$ in Theorem \ref{T: full subquiver}. Let \[\pReg=\bigcup_{n\in\NN_0}\pReg(n).\] We begin with a definition.

\begin{defn}\label{D:pullbackpar} Let $s\in\NN$. Define a function $-^{\#s}:\pReg\to\pReg$ by, for each $p$-regular partition $\mu$, let $\mu^{\#s}$ be the $p$-regular partition such that $\mu^{\#s}\sim_p\mu\cont (s)$.
\end{defn}

\begin{eg}\
\begin{enumerate}
  \item If $\lambda(\mu\cont (s))$ is $p$-regular, then $\mu^{\#s}=\lambda(\mu\cont (s))$.
  \item Let $p=2$, $\mu=(3,2,1)$ and $s=1$. Then $\mu^{\#s}=(4,3)$.
  \item Let $p=3$, $\mu=(3,3,1,1)$ and $s=1$. Then $\mu^{\#s}=(9)$.
\end{enumerate}
\end{eg}

To prepare for the proof of Theorem \ref{T: pullback simple}, we need the next two lemmas.

\begin{lem}\label{L: simple min} Let $\cO$ be either $\Z$ or $\F$, $\lambda\in\pReg(n)$ and  $0\neq v\in M_{\lambda,\cO}$. Then \[\lambda=\min_{\wref}\{\mu\in\pReg(n):\Xi^\mu v\neq 0\}.\] 
\end{lem}
\begin{proof} We have $\Xi^\mu v=\varphi^{\mu,\cO}(\lambda)v$. If $\varphi^{\mu,\cO}(\lambda)\neq 0$, then $\lambda\wref\mu$ by Lemma \ref{L: phi nonzero}(i).  By part (ii) of the same lemma, $\varphi^{\lambda,\cO}(\lambda)\neq 0$.
\end{proof}

\begin{lem}\label{L: pullback Z} Let $n\geq s\geq 1$ and $\mu\in\P(n)$. Then $\Fun _s(M_{\mu,\Z})\cong M_{\lambda,\Z}$ where $\lambda=\lambda(\mu\cont (s))$.
\end{lem}
\begin{proof} Let $\Fun _s(M_{\mu,\Z})\cong M_{\lambda,\Z}$ be spanned by $v$. Through the map $\Delta_s$, we have
\begin{align*}z:=\Xi^{\mu\cont (s)}\cdot v=\varphi^{\mu}(\mu)v+\sum_{j\in [1,\ell(\mu)]} \varphi^{\mu^{(j)}\cont (s)}(\mu)v.
\end{align*} Observe that $z$ is a nonzero vector as $\varphi^{\mu}(\mu)> 0$ and $\varphi^{\mu^{(j)}\cont (s)}(\mu)\geq 0$. By Lemma \ref{L: simple min}, $\lambda\wref \mu\cont (s)$. On the other hand, since $\Xi^\lambda\cdot v=\varphi^{\lambda}(\lambda)v\neq 0$, there exists $k\in [1,\ell(\lambda)]$ such that $\varphi^{\lambda^{(k)}}(\mu)\neq 0$ and hence $\mu\wref \lambda^{(k)}$ by Lemma \ref{L: phi nonzero}(i). In particular, we must have $\ell(\lambda^{(k)})\leq \ell(\mu)<\ell(\lambda)$. This forces $\ell(\lambda^{(k)})=\ell(\mu)=\ell(\lambda)-1$ and hence $\ell(\lambda)=\ell(\mu\cont (s))$. Therefore $\lambda\approx \mu\cont (s)$.
\end{proof}

We are now ready to state and prove the  main result of this section.

\begin{thm}\label{T: pullback simple} Let $s\in\NN$, $n\geq s$ and $\mu\in\pReg(n)$. Then $\Fun _s(M_{\mu,\F})\cong M_{\mu^{\#s},\F}$.
\end{thm}
\begin{proof} Let $M_{\mu,\Z}$ be $\Z$-spanned by $v$. Consider both $\Fun _s(M_{\mu,\Z}\otimes_\Z \F)$ and $\Fun _s(M_{\mu,\Z})\otimes_\Z \F$ as the same vector space as $M_{\mu,\Z}\otimes_\Z \F$ which is $\F$-spanned by $v\otimes 1$. For any composition $\eta$ of $n+s$, we have
\begin{align*}
  \Xi^\eta\cdot (v\otimes 1)&=\sum_{j\in[1,\ell(\eta)]}\Xi^{\eta^{(j)}} (v\otimes 1)
  =\sum_{j\in[1,\ell(\eta)]}\varphi^{\eta^{(j)},\F}(\mu) (v\otimes 1)\\
  &=\sum_{j\in[1,\ell(\eta)]}(\varphi^{\eta^{(j)}}(\mu)v)\otimes 1=(\Xi^\eta\cdot v)\otimes 1.
\end{align*} This shows that $\Fun _s(M_{\mu,\Z}\otimes_\Z \F)\cong \Fun _s(M_{\mu,\Z})\otimes_\Z \F$. By Lemma \ref{L: pullback Z}, we have $\Fun _s(M_{\mu,\Z})\otimes_\Z \F\cong M_{\lambda(\mu\cont (s)),\Z}\otimes_\Z \F= M_{\mu^{\#s},\Z}\otimes_\Z \F\cong M_{\mu^{\#s},\F}$.
\end{proof}

As a consequence, we get the following corollary.

\begin{cor}\label{C: s inj} The map $-^{\#s}$ is injective.
\end{cor}
\begin{proof} Suppose that $\mu^{\#s}=\eta^{\#s}$ where $\mu,\eta\in\P_p(n)$. By Theorem \ref{T: pullback simple}, $\Fun _s(M_{\mu,\F})\cong M_{\mu^{\#s},\F}= M_{\eta^{\#s},\F}\cong\Fun _s(M_{\eta,\F})$. Thus, we have $M_{\mu,\F}\cong M_{\eta,\F}$, i.e., $\mu=\eta$.
\end{proof}

Since $\Delta_s:\Des{n}{\F}\to \Des{n-s}{\F}$ is surjective, by Lemma \ref{L:quiver}, we knew that the Ext quiver of $\Des{n-s}{\F}$ is a subquiver of the Ext quiver of $\Des{n}{\F}$. Theorem \ref{T: pullback simple} now shows how to identify the vertices in $Q_{n-s,p}$ as vertices in $Q_{n,p}$ using the map $-^{\#s}$. Therefore, we obtain the following theorem.

\begin{thm}\label{T: full subquiver} The quiver $Q_{n-s,p}$ is a subquiver of $Q_{n,p}$ via the identification which a vertex $\mu$ in $Q_{n-s,p}$ is identified with the vertex $\mu^{\#s}\in Q_{n,p}$.
\end{thm}


Explicit idempotents may be of independent interest. The following describes the effect of the algebra homomorphism $\Delta_s$ on a particular full set of orthogonal primitive idempotents. For this, we remind the reader about the algebra homomorphism $\theta$ in Theorem \ref{T: Sol epi}.

\begin{lem}\label{L: Delta idem} Let $\{e_{\lambda,\F}:\lambda\in\pReg(n)\}$ be a complete set of primitive orthogonal idempotents of $\Des{n}{\F}$ such that $\sum_\lambda e_{\lambda,\F}=1$ and $\sol(e_{\lambda,\F})=\Char_{\lambda,\F}$ and let \[\Upsilon=\{\mu^{\#s}:\mu\in\pReg(n-s)\}.\] Then $\Delta_s(e_{\lambda,\F})=0$ if $\lambda\not\in\Upsilon$ and  $\{\Delta_s(e_{\lambda,\F}):\lambda\in \Upsilon\}$ is a complete set of primitive orthogonal idempotents of $\Des{n-s}{\F}$ such that $\sum_{\lambda\in\Upsilon} \Delta_s(e_{\lambda,\F})=1$ and $\sol(\Delta_s(e_{\lambda,\F}))=\Char_{\mu,\F}$ where $\mu^{\#s}=\lambda$.
\end{lem}
\begin{proof} Apply the algebra homomorphism $\Delta_s$, we have that $f_\lambda:=\Delta_s(e_{\lambda,\F})$ is 0 or an idempotent, $f_\lambda f_\eta=0$ if $\lambda\neq \eta$ and $\sum_{\lambda}f_\lambda=1$. For any $\mu\in\pReg(n-s)$, by Theorem \ref{T: pullback simple}, we have $\Fun _s(M_{\mu,\F})\cong M_{\mu^{\#s},\F}$. Let $0\neq v\in M_{\mu^{\#s},\F}$. Then  \[\delta_{\lambda,\mu^{\#s}}v=\Char_{\lambda,\F}(\mu^{\#s})v=e_{\lambda,\F}\cdot v=f_\lambda v=\sol(f_\lambda)(\mu)v.\] This implies that \[\sol(f_\lambda)(\mu)=\left \{\begin{array}{ll} 1&\text{if $\mu^{\#s}=\lambda$,}\\ 0&\text{otherwise.}\end{array}\right .\] Since $-^{\#s}$ is injective by Corollary \ref{C: s inj}, we have $\sol(f_\lambda)=\Char_{\mu,\F}$ if $\mu^{\#s}=\lambda$. Also, since $\pReg(n-s)$ (or $\Upsilon$ by Corollary \ref{C: s inj}) labels the primitive orthogonal idempotents of $\Des{n-s}{\F}$ and we already get a complete set of primitive orthogonal idempotents $\{f_\lambda:\lambda\in \Upsilon\}$, we  must have $f_\lambda=0$ if $\lambda\not\in\Upsilon$.
\end{proof}

\bigskip

We end this section with an example illustrating  Theorem \ref{T: full subquiver}

\begin{eg} The quiver $Q_{5,2}$ can be seen as a subquiver (in blue) of $Q_{6,2}$ as below.
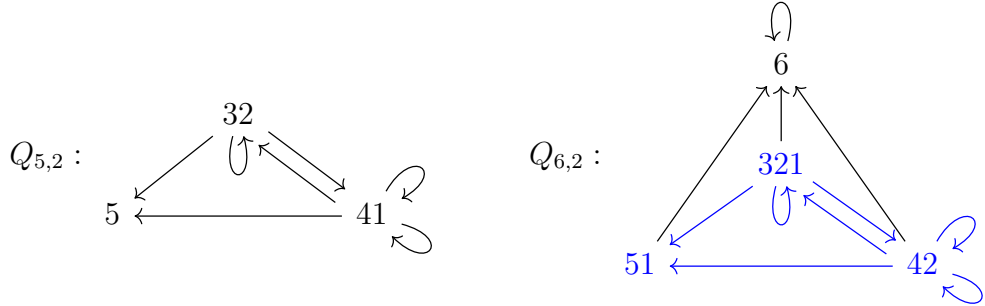
\[Q_{5,2}:\begin{tikzcd}
    &32\arrow[out=255,in=285,loop,distance=20pt]\arrow[dr,shift left=.4ex]\ar[dl]&\\
    5&&41\arrow[out=345,in=315,loop,distance=20pt]\arrow[out=60,in=30,loop,distance=20pt]\arrow[ul,shift left=0.8ex]\ar[ll]
\end{tikzcd} \hspace{1cm}
Q_{6,2}:\begin{tikzcd}
    &6\arrow[out=75,in=105,loop,distance=20pt]&\\
    &{\tbl{321}\arrow[out=255,in=285,loop,distance=20pt,color=blue]\arrow[dr,shift left=.4ex,color=blue]\ar[u]\ar[dl,color=blue]}&\\
    \tbl{51}\ar[uur]&&\tbl{42}\arrow[out=345,in=315,loop,distance=20pt,color=blue]\arrow[out=60,in=30,loop,distance=20pt,color=blue]\arrow[ul,shift left=0.8ex,color=blue]\ar[uul]\ar[ll,color=blue]
\end{tikzcd}\]
\end{eg}

\section{Proof of Theorem \ref{T: rep type}}\label{S: rep type}

In this section, we prove our main result Theorem \ref{T: rep type}. When $p=\infty$, the representation type of the descent algebras of type $\A$  has been studied earlier by Schocker \cite{Schocker:2004a} in which he proved that $\Des{n}{\F}$ is finite type if and only if $n\leq 5$ and wild otherwise. Our main result in this section deals with the $p<\infty$ case. In particular, the $p=\infty$ case may be seen as the asymptotic result of Theorem \ref{T: rep type} when $p\to\infty$ (this is a reason why we use the convention of $p=\infty$ instead of $p=0$).

For the remainder of this section, we assume that the characteristic $p$ of the field $\F$ is finite, $\cO$ is a local principal ideal domain with the maximal ideal $(\pi)$, $K$ is the field of fractions of $\cO$ and $\FF= \cO/(\pi)$. Recall that $Q_{n,p}$  denotes the Ext quiver of $\Des{n}{\F}$.
\bigskip

For easy reference, we restate our main result here:

\begin{mainthm} Assume $F$ is a field of characteristic $p<\infty$. The descent algebra $\Des{n}{\F}$ has finite representation type if and only if either
\begin{enumerate}[(i)]
  \item $p=2$ and $n\leq 3$,
  \item $p=3$ and $n\leq 4$, or
  \item $p\geq 5$ and $n\leq 5$.
\end{enumerate} Otherwise, it has wild representation type.
\end{mainthm}

The proof of Theorem \ref{T: rep type} is obtained by using Corollary \ref{C: BGR hom} and the series of lemmas in this section below. Basically, for each prime $p$, we need to find the integer $n$ such that $\Des{n}{\F}$ and $\Des{n+1}{\F}$ have finite and wild types respectively.  To prove the series of lemmas, we shall combine the use of Cartan and decomposition matrices and partial information about the Ext quivers. In particular, in the next lemma, we provide two different proofs to demonstrate how these methods may be employed.

\begin{lem}\label{L: p2n3} Let $p=2$. The descent algebra $\Des{3}{\F}$ is finite type.
\end{lem}
\begin{proof} By Theorem \ref{T: Sol epi}, $\rad(\Des{3}{\F})$ is spanned by $u=\Xi^{111}$ and $v=\Xi^{21}-\Xi^{12}$. Since $\rad^2(\Des{3}{\F})=0$, $\Des{3}{\F}$ is two-nilpotent. Using the idempotents $e_{3,\F}=\Xi^3+\Xi^{21}+\Xi^{111}$ and $e_{21,\F}=\Xi^{21}+\Xi^{111}$ (see \cite[Appendix A]{Lim:2023a}), it is straightforward to check that $e_{21,\F}ue_{21,\F}=u$, $e_{3,\F}ve_{21,\F}=v$ and hence $Q_{3,2}$ is
\[\begin{tikzcd}
    3&21\arrow[l]\arrow[out=-15,in=15,loop,distance=20pt]
  \end{tikzcd}\]  The separated quiver of $Q_{3,2}$ is
\[\tiny{\begin{tikzcd}
    \circ&\circ&\circ\ar[l]\ar[r]&\circ
\end{tikzcd}}\] As such, $\Des{3}{\F}$ is finite type using Theorem \ref{T:separatedquiver}(i).

Since we also know both the Cartan and decomposition matrices of the descent algebras, there is an alternative proof. The Ext quiver $Q:=Q_{3,\infty}$ is
\[\begin{tikzcd}
    3&21\arrow[l]&1^3
\end{tikzcd}\] By Corollary \ref{C:nleq5}, $\Des{3}{K}\cong KQ$. With the lexicographic
order $(1^3) < (2,1) < (3)$, the Cartan matrix $C$ of $KQ$, the decomposition matrix $D$ with $p=2$ and the Cartan matrix $\wt C$ of $\Des{3}{\F}$ are, respectively,
\begin{align*}
C &= \begin{pmatrix} 1 & 0 & 0\cr 0 & 1 & 1\cr 0 & 0 & 1\end{pmatrix},&
D &=\begin{pmatrix} 1 & 0\cr 1 & 0 \cr 0 & 1\end{pmatrix},&
\wt C&=D^\top CD=\begin{pmatrix} 2 & 1 \cr 0 & 1\end{pmatrix}.
\end{align*} By Corollary \ref{C:Ext1Descent}, there is an arrow $3\leftarrow 21$ in $Q_{3,2}$. From $\wt C$, we see that $M_{3,\F}$ is projective and deduce that $\rad(P_{21,\F})$ is the direct sum of $M_{3,\F}$ and $M_{21,\F}$ (as $\rad(P_{21,\F})$ surjects onto $M_{3,\F}$) and hence there is a loop at the vertex $21$ of $Q_{3,2}$. This shows that the algebra is two-nilpotent. We now proceed as before by checking the separated quiver that the algebra is finite type.
\end{proof}

\begin{lem}\label{L:p2n4} Let $p=2$. The descent algebra $\Des{4}{\F}$ is wild.
\end{lem}
\begin{proof} Let $e=e_{4,\F}=\Xi^{4} + \Xi^{31} + \Xi^{21^2} + \Xi^{1^4}$ (see \cite[Appendix A]{Lim:2023a}), $u=\Xi^{22}+\Xi^{211}+\Xi^{121}$ and $v=\Xi^{211}$. It is easy to check that the local algebra $A:=e\Des{4}{\F}e$ has a basis $\{e,u,v,u^2,uv\}$. In its radical, we have \[\begin{array}{c|cccc} &u&v&u^2&uv\\ \hline u&u^2&uv&0&0 \\ v&0&u^2&0&0\\ u^2&0&0&0&0\\ uv&0&0&0&0\end{array}\] It is clear that $A$ is isomorphic to the algebra $F\langle x,y\rangle/(x^2+y^2,yx)$ by identifying $x,y$ with $u,v$ respectively. By \cite[(1.2)]{Rin488}, the algebra $A$ is wild. Therefore $\Des{4}{\F}$ is wild by Lemma \ref{L: surj alg}(ii).
\end{proof}

Lemmas \ref{L: p2n3} and \ref{L:p2n4} settle the $p=2$ case for Theorem \ref{T: rep type}. We now deal with the $p=3$ case.

\begin{lem}\label{L:p3n4} Let $p=3$. The descent algebra $\Des{4}{\F}$ is finite type and its Ext quiver is \[\begin{tikzcd}
    4&31\arrow[l]\arrow[out=75,in=105,loop,distance=20pt]&21^2\arrow[l]&2^2.
\end{tikzcd}\]
\end{lem}
\begin{proof} The Ext quiver $Q=Q_{4,\infty}$ of $\Des{4}{K}$ is
\[\begin{tikzcd}
    4&31\arrow[l]&21^2\arrow[l]&2^2&1^4
\end{tikzcd}\] By Corollary \ref{C:nleq5}, we know that $\Des{4}{K}\cong KQ$. We order the partitions of $4$ with respect to the lexicographic order by $1^4<21^2<2^2<31<4$. Since $p=3$, we have $1^4\sim_3 31$. Let $C$ be the Cartan matrix of $\Des{4}{K}$ and $D$ be the decomposition matrix for $p=3$. As such, by Theorem \ref{T: APW},
\[\wt C=D^\top CD=\begin{pmatrix}1 & 0 & 1 & 1\cr 0 & 1 & 0 & 0\cr 0 & 0 & 2 & 1\cr 0 & 0 & 0 & 1\end{pmatrix}.\] By Corollary \ref{C:Ext1Descent}, we have arrows $4\leftarrow 31$ and $31\leftarrow 21^2$ in $Q_{4,3}$. We claim that
\begin{enumerate}[(a)]
\item $\rad(P_{31,\F})$ is semisimple the direct sum of $M_{4,\F}$ and $M_{31,\F}$, and
\item $P_{21^2,\F}$ is the injective hull of $M_{4,\F}$.
\end{enumerate}

For (a), by the Cartan matrix $\wt C$, notice that $M_{4,\F}$ is projective. Since there is an arrow $4\leftarrow 31$ in $Q_{4,3}$, there is a surjection $\rad(P_{31,\F})\to M_{4,\F}$ and it must split. By the Cartan matrix again, we see that $\rad(P_{31,\F})\cong M_{4,\F}\oplus M_{31,\F}$. For (b), the projective module $P_{21^2,\F}\cong P_{21^2,\cO}\otimes_\cO\F$ is uniserial with composition factors, from top to bottom, $M_{21^2,\F}$, $M_{31,\F}$ and $M_{4,\F}$ by Lemma \ref{L:projmodp}(iii). The injective hull of $M_{4,\F}$ has the same composition factors as $P_{21^2,\F}$ which can be seen from the last column of the Cartan matrix $\wt C$. Hence $P_{21^2,\F}$ is actually isomorphic to the injective hull of $M_{4,\F}$ and thus $P_{21^2,\F}$ is both injective and projective. We have obtained the Ext quiver $Q_{4,3}$ of $\Des{4}{\F}$ is given as in the statement.

We could make use of Lemma \ref{L: surj alg}(iii) to conclude that $B:=\Des{4}{\F}/\soc(P_{21^2,\F})=\Des{4}{\F}/M_{4,\F}$ has the same representation type as $\Des{4}{\F}$. But  it is elementary in our case which we shall give a brief argument. Let $\omega$ span the socle of the projective-injective $\Des{4}{\F}$-module $P:=P_{21^2,\F}$. If $M$ is an indecomposable $\Des{4}{\F}$-module and $\omega m\neq 0$ for some $m\in M$ then the module $\Des{4}{\F}m$ is isomorphic to $P$ and is a submodule of $M$. Since $P$ is injective, we must have $M\cong P$. That is, all indecomposable $\Des{4}{\F}$-modules except $P$ are $B$-modules. This shows that $\Des{4}{\F}$ is finite type if $B$ is finite type.

The algebra $B$ is two-nilpotent and has the same Ext quiver $Q_{4,3}$. The separated quiver of $Q_{4,3}$ is a disjoint union of Dynkin diagrams of type $\A$ and thus $B$ is finite type by Theorem \ref{T:separatedquiver}(i).
\end{proof}

\begin{rem} It can also be shown that, when $p=3$, $\Des{4}{\F}$ is a string algebra (see \cite[\S II]{Erd90}).
\end{rem}

Consider $\Des{5}{\F}$. If $p\geq 7$, the Ext quiver $Q_{5,p}$ of $\Des{5}{\F}$ is
\begin{equation}\label{Dg:Q50}
\begin{tikzcd}
    2^21\arrow[r]&32\arrow[r]&5&41\arrow[l]&31^2\arrow[l]&21^3\arrow[l]&1^5
\end{tikzcd}
\end{equation}
In this case, $\Des{5}{\F}\cong \F Q_{5,p}$ by Corollary \ref{C:nleq5} and we can write down its Cartan matrix $C$ with respect to the lexicographic order
$1^5<21^3 < 2^21 < 31^2<32<41<5$ as follows.
\begin{equation}\label{Eq:C50}
C=\begin{pmatrix} 1 & 0&0&0&0&0&0\cr
0 & 1 & 0 & 1 & 0 & 1 & 1\cr
0&0&1&0&1&0&1\cr
0&0&0 &1 & 0 &1&1 \cr
0&0&0&0& 1 & 0 & 1 \cr
0&0&0&0&0&1 & 1 \cr
0&0&0&0&0&0&1\end{pmatrix}
\end{equation}

\begin{lem}\label{L:p3n5} Let $p=3$. The descent algebra $\Des{5}{\F}$ is wild.
\end{lem}
\begin{proof} We claim that the Ext quiver $Q_{5,3}$ of $\Des{5}{\F}$ contains the following subquiver (cf. Remark \ref{R:Q53}).
\begin{equation}\label{Dg:Q53}
\begin{tikzcd}
   5&32\ar[l]\ar[out=75,in=105,loop,distance=20pt]\ar[d]&221\ar[l]\\
   41\ar[u]&{311\ar[l]\ar[out=15,in=-15,loop,distance=20pt]}
\end{tikzcd}
\end{equation}  Once we have proved the claim, since $\Des{5}{\F}/\rad^2(\Des{5}{\F})$ is two-nilpotent and has the same Ext quiver $Q_{5,3}$, we see that the separated quiver of $Q_{5,3}$ contains the following subquiver \[\tiny{\begin{tikzcd}
    \circ\ar[r]&\circ&\circ\ar[l]\ar[r]\ar[d]&\circ&\circ\ar[l]\ar[r]&\circ\\ &&\circ\\&&\circ\ar[u]
\end{tikzcd}}\] which is neither a union of Dynkin quivers nor a union of extended Dynkin quivers. So $\Des{5}{\F}$ is wild by Lemma \ref{L: surj alg}(i) and Theorem \ref{T:separatedquiver}.

Using Corollary \ref{C:Ext1Descent} and $Q_{5,\infty}$, we get all the arrows in the quiver (\ref{Dg:Q53}) except the loops at both vertices $32$ and $311$. The loop at $311$ can be obtained using Theorem \ref{T: full subquiver} and Lemma \ref{L:p3n4}. For the loop at $32$, we compute the Cartan matrix for $p=3$ using (\ref{Eq:C50}) and get
$$\wt{C} = D^\top CD=\left(\begin{matrix} 1 & 0&1 & 0 & 1\cr 0&2&0&1&1\cr 0&1&2&1&2\cr 0&0&0&1&1\cr 0&0&0&0&1\end{matrix}\right).
$$  Let $T=P_{21^3, \cO}\otimes_\cO \F$. By $\wt{C}$ and Lemma \ref{L:projmodp}(iii), $T$ is uniserial and there is a short exact sequence $$0\to W \to P_{32,\F} \to  T\to 0
$$ for some $W$ with composition factors $M_{32,\F}$ and $M_{5,\F}$. Notice that $(\rad(P_{32,\F}):M_{32,\F})=1$. Suppose on the contrary that $M_{32,\F}$ does not appear in \[Z:=\rad(P_{32,\F})/\rad^2(P_{32,\F}).\] As such, $Z\cong M_{31^2,\F}$ or $Z\cong M_{31^2,\F}\oplus M_{5,\F}$. This is a contradiction since neither $P_{31^2,\F}$ nor $P_{5,\F}$ contains $M_{32,\F}$ as a composition factor. Thus there is a loop at $32$.
\end{proof}

Lemmas \ref{L:p3n4} and \ref{L:p3n5} settle the $p=3$ case for Theorem \ref{T: rep type}. We now focus on the $p\geq 5$ case. The case $p=5=n$ is more involved and will be dealt with last. We begin with the $p\geq 7$ case.

\begin{lem} Let $p\geq 7$. The descent algebra $\Des{5}{\F}$ is finite type.
\end{lem}
\begin{proof}  By Corollary \ref{C:nleq5}, $\Des{5}{\F}\cong \F Q_{5,p}$. Since the underlying graph of $Q_{5,p}$ (see the quiver (\ref{Dg:Q50})) is a disjoint union of Dynkin diagrams of type $\A$, by Theorem \ref{T: Gab}, $\Des{5}{\F}$ is hereditary and finite.
\end{proof}

\begin{lem} Let $p\geq 5$. The descent algebra $\Des{6}{\F}$ is wild.
\end{lem}
\begin{proof} Suppose first that $p\geq 7$. By the quiver (\ref{Dg:Q60}), $Q_{6,p}$ contains the subquiver (see \cite[Figure 3]{Schocker:2004a}) \[
\begin{tikzcd}&&2211\ar[d]\\ 411\ar[r]&51&321\ar[l]\ar[d]\ar[r]&33\\ &&42
\end{tikzcd}\]
This is a wild quiver, however we can only deduce that
$\Des{6}{\F}$ is wild if we know that the paths in this quiver are not involved in any relation. For the $p=\infty$ case, this was stated in \cite{Schocker:2004a}. Let $B=\Des{6}{\F}/(e_{6,\F})$, i.e., we factor out the ideal generated by the idempotent $e_{6,\F}$. Its Ext quiver $Q$ is obtained from $Q_{6,p}$ by removing the vertex $6$ and the two arrows ending at $6$. By Corollary \ref{C:n6}, $B\cong FQ_{6,p}/(e_{6,\F},\omega)$ where $\omega=e_{6,\F}\omega e_{2^21^2,\F}$ and thus it is isomorphic to the path algebra $FQ$. By Theorems \ref{T: Gab} and \ref{T:tame}, $B$ and hence $\Des{6}{\F}$ are wild. This completes the proof for $p\geq 7$.

We assume now $p=5$ and claim that $Q_{6,5}$ contains the following subquiver (cf. Remark \ref{R:Q53}).
\begin{equation}\label{Dg:Q65}
\begin{tikzcd}
    6&51\ar[l]\ar[out=75,in=105,loop,distance=20pt]&41^2\ar[l]&31^3\ar[l]&21^4\ar[l]\\ 42\ar[u]&321\ar[u]\ar[l]\ar[r]&3^2&2^3\\ &2^21^2\ar[u]
\end{tikzcd}
\end{equation} By Corollary \ref{C:Ext1Descent} and the quiver (\ref{Dg:Q60}), we obtain all the arrows in quiver (\ref{Dg:Q65}) except the loop at $51$. Let $\wt{C}$ be the Cartan matrix of $\Des{6}{\F}$ as in (\ref{Eq:C60}). Since $M_{6,\F}$ is projective, argue as in the proof of Lemma \ref{L:p3n4} (for claim (a)), we conclude that there is a loop at $51$. As such, the separated quiver of the Ext quiver of $B=\Des{6}{\F}/\rad^2(\Des{6}{\F})$ contains the subquiver:
\[\tiny{\begin{tikzcd}
    \circ\ar[r]&\circ&\circ\ar[l]\ar[r]&\circ&\circ\ar[l]\\
    &&\circ&\circ\ar[u]\ar[l]\ar[r]&\circ
\end{tikzcd}}\] We conclude that $\Des{6}{\F}$ is wild by Lemma \ref{L: surj alg}(i) and Theorem \ref{T:separatedquiver}.
\end{proof}

We are left with the case when $p=5$ and $n=5$ and will first determine a presentation of the algebra $\Des{5}{\F}$.


\begin{lem}\label{L: p5n5} Let $p=5$ and $Q$ be the quiver
\begin{equation}\label{Dg:Q55}\begin{tikzcd}
   1\ar{r}{\alpha}&2\ar{r}{\beta}&3\ar{r}{\gamma}&4\ar[out=105,in=75,loop,distance=20pt,"\varepsilon"]&5\ar{l}[swap]{\delta}&6\ar{l}[swap]{\eta}
  \end{tikzcd}
\end{equation} Then $\Des{5}{\F}$ is isomorphic to $\F Q/I$ where $I$ is generated by $\varepsilon^2$, $\varepsilon\gamma$, $\varepsilon\delta$.
  \end{lem}
\begin{proof} The Ext quiver and Cartan matrix $C$ of $\Des{5}{K}$ are given in the quiver (\ref{Dg:Q50}) and (\ref{Eq:C50}) respectively. Using Theorem \ref{T: APW}, the Cartan matrix of $\Des{5}{\F}$ is obtained from $C$ by replacing the entry $C_{(5),(5)}$ by $2$ and then removing the first row and column. By Corollary \ref{C:nleq5}, we see that any indecomposable projective $\Des{5}{K}$-module $P_{\lambda,K}$ is uniserial. For $\lambda\neq (5)$, the submodule structure of $P_{\lambda,\F}$ is the same as for infinite characteristic. On the other hand, $P_{5,\F}$ has two composition factors $M_{5,\F}$. These are enough to check that the presentation is as stated (with the vertices $1,2,3,4,5,6$ identified with the simple modules labelled by $21^3,31^2,41,5,32,2^21$ respectively).
\end{proof}

\begin{rem}\label{R:Q53} Using the construction of the idempotents in \cite{Lim:2023a} and Magma \cite{Magma}, one may check that the quivers (\ref{Dg:Q53}), (\ref{Dg:Q65}) and (\ref{Dg:Q55}) are precisely $Q_{5,3}$, $Q_{6,5}$ and $Q_{5,5}$ respectively.
\end{rem}

Let $R$ be the following quiver of type $\mathbb{E}_7$ for the rest of this section:
\begin{equation}\label{Dg:E7}
\begin{tikzcd}
   &&&7\ar{d}{\rho}\\1\ar{r}{\alpha}&2\ar{r}{\beta}&3\ar{r}{\gamma}&4&5\ar{l}[swap]{\delta}&6\ar{l}[swap]{\eta}
\end{tikzcd}
\end{equation} For each $i\in[1,7]$, let $\f_i$ be the idempotent of the path algebra $FR$ labelled by the vertex $i$. Notice that $\dim_\F \F R=17$.

\begin{lem}\label{L: p5n5v2} Let $B=FR$ be the path algebra and $A$ be the subalgebra  of $B$ generated by the idempotents $\f_i$ where $i\neq 4,7$, $\hat{\f}_4:=\f_4+\f_7$ and all arrows of lengths at least 1. The algebra $A$ has codimension 1 in $B$ and $\rad(A)=\rad(B)$. Moreover, when $p=5$, we have $\Des{5}{\F}\cong A$.
\end{lem}
\begin{proof} It is easy to check that $\dim_\F A=16=\dim_\F \Des{5}{\F}$. The radical of $A$ contains $\alpha,\beta,\gamma,\delta,\eta,\rho$ and hence all paths of $B$ of positive length. Thus $\rad(A)=\rad(B)$. Let $\varepsilon:=\hat{\f}_4\rho\hat{\f}_4=\rho$ and notice that $\gamma=\hat{\f}_4\gamma \f_3$ and $\delta=\hat{\f}_4\delta \f_5$. When $p=5$, we identify $\Des{5}{\F}$ with $\F Q/I$ using Lemma \ref{L: p5n5} where $Q$ is the quiver (\ref{Dg:Q55}). The quiver of $A$ can be identified with $Q$ in the obvious way. Next, one checks that $\varepsilon^2=0$, $\varepsilon\gamma=0$ and $\varepsilon\delta=0$ which are the relations defining $\Des{5}{\F}$. As such, we have an algebra homomorphism from $\Des{5}{\F}$ onto $A$ which is an isomorphism by comparing their dimensions.
\end{proof}

\begin{lem}\label{L:p5n5v3} Let $p=5$. The descent algebra $\Des{5}{\F}$ is finite type.
\end{lem}
\begin{proof} We identify $\Des{5}{\F}$ with the subalgebra $A$ of $B$ as in Lemma \ref{L: p5n5v2}. We claim that the pair $A,B$ satisfies Lemma \ref{L:radicalemb}. Clearly, $\rad(A)=\rad(B)$. Also, we have an exact sequence
\begin{equation}\label{Eq:SESD55}
0\to A\to B\stackrel{\pi}\to B/A\to 0
\end{equation}
of left $A$ modules where $\pi$ is the canonical surjection. We claim that the sequence \ref{Eq:SESD55} splits as $A$-modules and hence $B\cong A\oplus S_4$ where $S_4$ is the simple module labelled by the vertex $4$ as in the quiver $R$ (see quiver (\ref{Dg:E7})).

The quotient $B/A$ has basis $\{\f_4+A\}$ and is simple isomorphic to $S_4$ as $A$-modules. Notice that $\f_4 = \hat{\f}_4\f_4$.  Define a linear map $\kappa: B/A\to B$ by
$$\kappa(\f_4+A):= \f_4 \in B$$
The $A$-submodule of $B$ generated by $\f_4$ is spanned by $\hat{\f_4}\f_4$. Notice $\rho \f_4=0$ and any other arrow is also killed by $\f_4$. So $\kappa$  is an $A$-module homomorphism. Since $\pi \kappa = {\rm Id}_{B/A}$, the claim follows.

By Theorem \ref{T: Gab}, $B$ is finite type. So, by Lemma \ref{L:radicalemb}, $A$ is finite type.
\end{proof}

\begin{rem} Consider the setting as in the proof of Lemma \ref{L:p5n5v3}. The path algebra $B$ of type $\mathbb{E}_7$ has $63$ non-isomorphic indecomposable $B$-modules which are in one-to-one correspondence with the positive roots of $\mathbb{E}_7$ by Theorem \ref{T: Gab}. Explicit calculation using the Auslander-Reiten theory shows that there are a total of $62$ non-isomorphic indecomposable $A$-modules. The only two indecomposable $B$-modules which are isomorphic upon restriction to $A$ are the simple $B$-modules labelled by the vertices $4,7$ as in the quiver (\ref{Dg:E7}).
\end{rem}

Putting everything together, we get the proof for our main theorem.

\begin{proof}[Proof of Theorem \ref{T: rep type}] Use Corollary \ref{C: BGR hom} and Lemmas \ref{L: p2n3}--\ref{L:p5n5v3}.
\end{proof}

Based on calculation, we end our paper with the following conjecture on the Ext quiver of $\Des{n}{\F}$. For this purpose, we introduce some notation.

For $\lambda\in\P(n)$, we write $(\lambda,p)=1$ if $p\nmid \lambda_i$ for all $i\in[1,\ell(\lambda)]$, otherwise we write $(\lambda,p)\neq 1$. Let $Q=(Q_0,Q_1)$ be a quiver. For any two vertices $v,w\in Q_0$, we write $\am^Q_{v,w}$ for the number of arrows (respectively, loops if $v=w$) from $v$ to $w$, that is, \[\am^Q_{v,w}=|\{\gamma\in Q_1:s(\gamma)=v,\ t(\gamma)=w\}|.\]

\begin{conj}\label{Conj: Ext Quiv} Let $\lambda,\mu\in\pReg(n)$. In the Ext quiver $Q_{n,p}$ of the descent algebra $\Des{n}{\F}$,
\begin{enumerate}[(i)]
  \item if $\lambda=\mu$, then $\am^{Q_{n,p}}_{\lambda,\lambda}>0$ if and only if $(\lambda,p)\neq 1$.
  \item if $\lambda\neq\mu$, there is an arrow $\lambda\to\mu$ in $Q_{n,p}$ if and only if there exist $\delta\sim_p\lambda$ and $\gamma\sim_p\mu$ such that there is an arrow $\delta\to \gamma$ in the Ext quiver $Q_{n,\infty}$. In this case, $\am^{Q_{n,p}}_{\lambda,\mu}=1$.
\end{enumerate}
\end{conj}

We remark that Corollary \ref{C:Ext1Descent} offers partial information for Conjecture \ref{Conj: Ext Quiv}(ii), that is, $\am^{Q_{n,p}}_{\lambda,\mu}\neq 0$ if there exist $\delta\sim_p\lambda$ and $\gamma\sim_p\mu$ such that $\am^{Q_{n,\infty}}_{\delta,\gamma}\neq 0$ (in this case, $\am^{Q_{n,\infty}}_{\delta,\gamma}=1$ as shown by Schocker (see Theorem \ref{T:OrdExtquiver})).

\begin{eg} It is not clear what the number of the loops at the vertices labelled by partitions $\lambda$ such that $(\lambda,p)\neq 1$ is. We demonstrate an example using Magma. Suppose that $p=2$ and we further simplify the notation by replacing $\am^{Q_{n,2}}_{\lambda,\lambda}$ with $\am_\lambda$ if $\lambda\in \Lambda^+_2(n)$. We have
\begin{align*}
\am_{2}&=1,& \am_{21}&=1,&\am_{4}&=2,& \am_{32}&=1,\\
\am_{41}&=2,&\am_{6}&=1,&\am_{42}&=2,& \am_{321}&=1,\\
\am_{421}&=2,&\am_{61}&=1,&\am_{52}&=1,&\am_{43}&=2,\\
\am_{62}&=2,&
\am_{8}&=3,&\am_{521}&=1,& \am_{431}&=2,
\end{align*} and $\am_\lambda=0$ if $\lambda\in\Lambda^+_2(n)$ where $n\in [1,8]$ and $(\lambda,2)=1$.
\end{eg}

\end{document}